\documentclass[a4paper,reqno]{amsart}

\usepackage[utf8]{inputenc}
\usepackage[english]{babel}
\usepackage[T1]{fontenc}
\usepackage{microtype}
\usepackage{amsmath,amssymb,amsthm}
\usepackage{thmtools,thm-restate}
\usepackage{mathtools}
\usepackage{upgreek}
\usepackage{orcidlink}
\usepackage{hyperref}
\hypersetup{hidelinks}
\hypersetup{breaklinks=true}
\usepackage[nameinlink]{cleveref}
\usepackage{suffix}
\usepackage{tikz}
\usetikzlibrary{cd}
\usetikzlibrary{calc}
\usepackage{pifont}
\usepackage{nicefrac}

\usepackage{enumitem}
\usepackage{amsaddr}

\title[{A quick guide to ordinary state-dependent DDEs}]{A quick guide to ordinary state-dependent delay differential equations}
\author{Bernhard Aigner\orcidlink{0009-0009-8252-162X}
  \and Marcus Waurick\orcidlink{0000-0003-4498-3574}}
\thanks{supported by the state of Saxony via a graduate student stipend}
\address{Institut f\"{u}r Angewandte Analysis\\
  TU Bergakademie Freiberg\\
  Pr\"{u}ferstr. 9, 09599 Freiberg\\
  Germany}
\email{bernhard.aigner@doktorand.tu-freiberg.de, marcus.waurick@math.tu-freiberg.de}
\date{\today}

\DeclarePairedDelimiterX{\norm}[1]{\lVert}{\rVert}{#1}
\DeclarePairedDelimiterX{\abs}[1]{\lvert}{\rvert}{#1}
\DeclarePairedDelimiterX{\scprod}[2]{(}{)}{#1\delimsize| #2}
\newcommand{\e}{\mathrm{e}} 
\newcommand{\dd}{\mathrm{d}} 
\newcommand{\dx}[1][x]{\,\dd{}#1}
\newcommand{\argdot}{\cdot}

\newtheorem*{definition*}{Definition}
\newtheorem*{theorem*}{Theorem}
\newtheorem*{proposition*}{Proposition}
\newtheorem*{lemma*}{Lemma}
\newtheorem*{corollary*}{Corollary}
\newtheorem*{remark*}{Remark}

\newtheorem{definition}{Definition}[section]
\newtheorem{theorem}[definition]{Theorem}
\newtheorem{proposition}[definition]{Proposition}
\newtheorem{lemma}[definition]{Lemma}

\newtheorem{remark}[definition]{Remark}

\begin{document}

\begin{abstract}
  We review $H^{1}$-well-posedness for initial value problems of ordinary differential equations with state-dependent right-hand side. We streamline known approaches to infer existence and uniqueness of solutions for small times given a Lipschitz-continuous prehistory. The paramount feature is a reduction of the differential equation to a fixed point problem that admits a unique solution appealing to the contraction mapping principle. The use of exponentially weighted Sobolev spaces in this endeavor proves to be as powerful as for ordinary differential equations without delay. Our result includes a blow-up criterium for global existence of solutions. The discussion of well-posedness is concluded by new results covering continuous dependence on initial prehistories and on the right-hand sides.
\end{abstract}

\maketitle

\section{Introduction}
In this article we recapitulate recent results on delay differential equations, in particular from our articles \cite{Waurick2023} and \cite {Aigner2024}. One of the main aims of the present article is to streamline the theory and present a largely self-contained account of the selected results. We include some previously unpublished theorems regarding continuous dependence on data. We start with a brief explication of the problem:\\
We want to investigate so-called ordinary state-dependent {\em delay differential equations (DDEs)}. As for ordinary differential equations (without delay) we will try to solve initial value problems (IVPs) of the following form
\begin{alignat}{3}
  \begin{aligned}
  \label{FDE}
  v'(t)&=G(t,v_{t})&&\text{for }\;\;\;\;\,0<t\leq T\\
  v&=\Phi&&\text{for }-h\leq t \leq 0\,\text{,}
  \end{aligned}
\end{alignat}
where we use the usual convention
\begin{equation*}
  \begin{split}
  v_{t}\colon\,[t-h,t]&\to H\\
    s&\mapsto v(t+s)\,\text{.}
  \end{split}
\end{equation*}
Differential equations as in \Cref{FDE} are called {\em ordinary} since the lhs is simply the time derivative (and no derivative shows up on the rhs either); a DDE because the dependence contained in the rhs $G$ is upon prior instances of the state variable. We investigate the very general case in which the dependence is upon an entire prior history of the state, so-called {\em state-dependent delay}.\\
In this article we aim to solve \Cref{FDE} in exponentially weighted Sobolev spaces. For this purpose we introduce the notation
\begin{equation*}
  \left(L_{2,\rho}(-h,T;H),\norm{\argdot}_{2,\rho}\right)\coloneq \left(L_{2}(-h,T;H),\e^{-2\rho t}\dx[t]\right)
\end{equation*}
and accordingly
\begin{align*}
  H^{1}_{\rho}(-h,T;H)&\coloneq \left\{u\in L_{2,\rho}(-h,T;H)\colon\,u'\in L_{2,\rho}(-h,T;H)\right\}\\
  \norm{u}_{H^{1}_{\rho}(-h,T;H)}&\coloneq \left(\norm{u}_{L_{2,\rho}(-h,T;H)}^{2}+\norm{u'}_{L_{2,\rho}(-h,T;H)}^{2}\right)^{\nicefrac{1}{2}}\,\text{,}
\end{align*}
where the time derivative is defined in the usual distributional sense. For an introduction to Sobolev spaces we refer to \cite[Ch.~5]{Evans2010}. We would like to point out that in this one-dimensial (in time) setting, the following result assures us that $H^{1}_{\rho}$-functions admit a continuous representative:
\begin{theorem}[Sobolev-embedding]
  \label{Sobolev}
  Let $-\infty<a<b<\infty$. Then $H^{1}(a,b;H)\hookrightarrow \mathcal{C}([a,b];H)$. More precisely, every $f\in H^{1}(a,b;H)$ admits an absolutely continuous representative and
  \begin{equation*}
    \norm{f}_{\mathcal{C}([a,b];H)}\leq \left((b-a)^{\nicefrac{1}{2}}+(b-a)^{-\nicefrac{1}{2}}\right)\norm{f}_{H^{1}(a,b;H)}
  \end{equation*}
\end{theorem}
This result can be found in any standard reference, e.g. \cite[Thm.~4.12]{Adams2003}. For a proof of the estimate we refer to \cite[Thm.~4.9]{Arendt2015}. In particular, the theorem implies that pointwise evaluation of $H_{\rho}^{1}$-functions is allowed. This will enable us to make use of the fundamental theorem of calculus later on.\\
Lastly we clarify what we mean by a solution to \Cref{FDE}. We refer to the following notion:
\begin{definition}
  \label{DefSol}
  A function $v\in H^{1}(-h,T;H)$ is called a {\em local solution} to \Cref{FDE} if there exists $0<T_{0}\leq T$ such that for all $0<T_{1}<T_{0}$ the restriction $v|_{[-h,T_{1}]}\in H^{1}(-h,T_{1};H)$ and
  \begin{alignat*}{4}
    v'(t)&=G(t,v_{t})&&\quad\text{for }&&{}\;\;\;\;\,0<t\leq T_{1}\\
    v&=\Phi&&\quad\text{for }&&-h\leq t \leq 0\,\text{.}
  \end{alignat*}
  If $T_{0}=T$ we call $u$ a {\em global solution}.
\end{definition}
We will state and prove our well-posedness results in section three, prior to that a preliminary section is necessary though. In section four, we will discuss several applications, followed by a brief discussion and contextualization of our findings in the concluding section.

\section{Preliminaries}
Before we can present our main results we need to devote some space to discuss important aspects of our approach. In particular we need to investigate the behaviour of the delay, understand the importance of the time derivative with boundary condition and estimate the evaluation functional.

\subsection{The delay}
In our aim to better understand DDEs, we particularly need to understand the behaviour of the history map, that is for a given $0\leq t\leq T$ the function
\begin{equation*}
  \begin{split}
    \Theta_{t} \colon \, H_{\rho}^{1}(-h,T;H)&\rightarrow H^{1}(-h,0;H)\\
    u&\mapsto u_{t}
  \end{split}
  \,\text{.}
\end{equation*}
However, it proves more useful to include the dependence on $u$ into the map and investigate
\begin{equation*}
  \begin{split}
    \Theta \colon H^{1}_{\rho}(-h,T;H)&\rightarrow L_{2,\rho}\left(0,T;H^{1}(-h,0;H)\right)\\
    u&\mapsto (t\mapsto \Theta_{t}u)
  \end{split}
  \,\text{.}
\end{equation*}
\begin{lemma}
  \label{ThetaLip}
  $\Theta$ is Lipschitz-continuous with $\norm{\Theta}\leq (2\rho)^{\nicefrac{-1}{2}}$.
\end{lemma}
\begin{proof}
  We first study $\Theta$ as a map from $L_{2,\rho}(-h,T;H)$ to $L_{2,\rho}(0,T;L_{2}(-h,0;H))$. For $u\in L_{2,\rho}\left(-h,T;H\right)$ we can estimate:
  \begin{align*}
    \norm{\Theta u}_{L_{2,\rho}(0,T;L_{2}(-h,0;H))}^{2}  &= \int_0^T \norm{\Theta_{t}u}_{L_{2}(-h,0;H)}^{2} \e^{-2\rho t} \dx[t] \\
                                                         &= \int_0^T \int_{-h}^0 \norm{u(s+t)}_{H}^{2} \dx[s]\, \e^{-2\rho t} \dx[t]\\
                                                         &= \int_0^T \int_{t-h}^t \norm{u(s)}_{H}^{2} \dx[s]\, \e^{-2\rho t} \dx[t]\\
                                                         &= \int_{-h}^T \int_s^{s+h} \norm{u(s)}_{H}^{2}\e^{-2\rho t} \dx[t] \dx[s]\\
                                                         &= \int_{-h}^T \norm{u(s)}_{H}^{2} \tfrac{1}{2\rho}(\e^{-2\rho s}-\e^{-2\rho (s+h)}) \dx[s]\\
                                                         &\leq \frac{1}{2\rho}\int_{-h}^T \norm{u(s)}_{H}^{2} \e^{-2\rho s} \dx[s]\\
                                                         &= \tfrac{1}{2\rho}\norm{u}_{L_{2,\rho}(-h,T;H)}^{2}
  \end{align*}
  Now we pass from the $L_{2}$-case to the $H^{1}$-case using the fact, that for $u \in H^{1}_{\rho}(-h,T;H)$ the weak derivative and $\Theta_{t}$ commute. We can estimate:
  \begin{align*}
    \norm{\Theta u}_{H^{1}_{\rho}(0,T;L_{2}(-h,0;H))}^{2}  &= \int_0^T \norm{\Theta_{t}u}_{L_{2}(-h,0;H)}^{2} \e^{-2\rho t}
                                                             + \norm{\Theta_{t}u'}_{L_{2}(-h,0;H)}^{2} \e^{-2\rho t} \dx[t]\\
                                                      &= \norm{\Theta u}_{L_{2,\rho}(0,T;L_{2}(-h,0;H))}^{2} +
                                                             \norm{\Theta u'}_{L_{2,\rho}(0,T;L_{2}(-h,0;H))}^{2}\\
                                                      &\leq \tfrac{1}{2\rho}\norm{u}_{L_{2,\rho}(-h,T;H)}^{2} +
                                                        \tfrac{1}{2\rho}\norm{u'}_{L_{2,\rho}(-h,T;H)}^{2} \\
                                                      &= \tfrac{1}{2\rho}\norm{u}_{H^{1}_{\rho}(-h,T;H)}^{2}\qedhere
  \end{align*}
\end{proof}
We point out the dependence on $\rho$ in the estimate. The use of exponentially weighted Sobolev spaces will allow us to increase the weights $\rho$ if necessary, resulting in a Lipschitz-constant of $\Theta$ as small as we desire. This is purely a technical advantage though, since on finite intervals, the  spaces $L_{2,\rho}$ and $H^{1}_{\rho}$ are isomorphic to $L_{2}$ and $H^{1}$ respectively.
\begin{remark}[Continuity with respect to time]
  \label{ThetaTimeCont}
  We can also investigate the dependence of the delay for some given fixed $u\in H^{1}_{\rho}\left(-h,T;H\right)$ with respect to time, i.e. the map $[0,T]\to H^{1}(-h,0;H),\,t\mapsto \Theta_{t}u$. We see from
  \begin{equation*}
    \norm{u_{t}-u_{s}}_{H^{1}(-h,0;H)}^{2} = \int_{-h}^0 \norm{u(t+r) - u(s+r)}_{H}^{2} + \norm{u'(t+r) - u'(s+r)}_{H}^{2} \dx[r]
  \end{equation*}
  that $u_{s}\to u_{t}$ in $H^{1}(-h,0;H)$ for $s\to t$. Hence the map $t\mapsto \Theta_{t}u$ is continuous.
\end{remark}

\subsection{The time derivative and its inverse}
Another advantage of utilizing the spaces $H^{1}_{\rho}$ is the fact, that we can define a realisation of the time derivative that proves to be boundedly invertible -- at least on a suitable subspace. To that end we define for $-\infty < a < b <+\infty$ the operator
\begin{equation*}
  \begin{split}
    I_{\rho}\colon\,L_{2,\rho}(a,b;H)&\to L_{2,\rho}(a,b;H)\\
    u&\mapsto \left(t\mapsto\int_{a}^{t} u(s)\dx[s]\right)\,\text{.}
  \end{split}
\end{equation*}
Then we can show:
\begin{proposition}
  $I_{\rho}$ is a bounded linear operator satisfying $\norm{I_{\rho}}\leq \frac{1}{\rho}$.
\end{proposition}
\begin{proof}
  For $u\in L_{2,\rho}(a,b;H)$ we can calculate:
  \begin{align*}
    \norm{I_{\rho}u}_{L_{2,\rho}(a,b;H)}^{2}
    &= \int_{a}^{b}\norm[\bigg]{\int_{a}^t u(s) \dx[s]}_{H}^{2} \e^{-2\rho t} \dx[t] \\
    &\leq \int_{a}^{b} \left(\int_{a}^t \norm{u(s)}_{H} \dx[s]\right)^{2} \e^{-2\rho t} \dx[t] \\
    &\leq \int_{a}^{b} \left(\int_{a}^t \norm{u(s)}_{H} \e^{-\nicefrac{\rho s}{2}} \e^{\rho (\nicefrac{s}{2}-t)}\dx[s]\right)^{2} \dx[t] \\
    &\leq \int_{a}^{b} \left(\int_{a}^t \norm{u(s)}_{H}^{2} \e^{-\rho s} \dx[s]\right) \left(\int_{a}^t \e^{\rho (s-2t)} \dx[s] \right) \dx[t] \\
    &\leq \frac{1}{\rho} \int_{a}^{b} \int_{a}^t \norm{u(s)}_{H}^{2} \e^{-\rho s} \dx[s]\,\e^{-\rho t}\dx[t] \\
    &= \frac{1}{\rho} \int_{a}^{b} \int_{s}^{b} \norm{u(s)}_{H}^{2} \e^{-\rho s} \e^{-\rho t} \dx[t]\dx[s] \\
    &\leq \frac{1}{\rho^{2}} \int_{a}^{b} \norm{u(s)}_{H}^{2} \e^{-2 \rho s} \dx[s] \\
    &= \tfrac{1}{\rho^{2}} \norm{u}_{L_{2,\rho}(a,b;H)}^{2} \qedhere
  \end{align*}
\end{proof}
\begin{remark}[range of $I_{\rho}$ and time derivative]
  The images of $I_{\rho}$ are weakly differentiable,  since the antiderivative (of an $L_{2}$-function on a bounded interval) is differentiable with the function as derivative, cf. \cite[Thm.~6.3.6]{Cohn1980}. As such, the images of $I_{\rho}$ are $H^{1}_{\rho}$-functions, that can be evaluated pointwise by \Cref{Sobolev}. Furthermore, the value at zero is zero. Therefore $I_{\rho}$ maps into
  \begin{equation*}
    H^{1}_{0,\rho}\left(-h,T;H\right)\coloneq \{u\in H^{1}_{\rho}\left(-h,T;H\right)\colon\,u|_{[-h,0]}=0\}
  \end{equation*}
  with operator norm $\norm{I_{\rho}}_{L_{2,\rho}\to H^{1}}\leq (1+\nicefrac{1}{\rho^{2}})^{\nicefrac{1}{2}}$, since for $u\in L_{2,\rho}(a,b;H)$:
  \begin{equation*}
    \norm{I_{\rho}u}_{H^{1}}^{2}=\norm{I_{\rho}u}_{L_{2,\rho}}^{2} + \norm{(I_{\rho}u)'}_{L_{2,\rho}}^{2}
    \leq \left(\norm{I_{\rho}}^{2}_{L_{2,\rho}\to L_{2,\rho}} + 1\right)\norm{u}_{L_{2,\rho}}^{2}
  \end{equation*}
  On the subspace with zero initial condition, $H^{1}_{0,\rho}\left(a,b;H\right)$, the time derivative now is just $\partial_{t}=\mathring{\partial_{t}}\coloneq I_{\rho}^{-1}$.
\end{remark}
In fact one can show surjectivity as well:
\begin{theorem}
  \label{domainDerivative}
  $H^{1}_{\rho}\left(a,b;H\right)\supseteq \operatorname{dom}\mathring{\partial_{t}}= H^{1}_{0,\rho}\left(a,b;H\right)$.
\end{theorem}
\begin{proof}
  That $\operatorname{dom}\mathring{\partial_{t}}\subseteq H^{1}_{0,\rho}\left(a,b;H\right)$ follows immediately from the definition. We only need to show that $I_{\rho}$ is onto to prove the reverse inclusion. For a given $u\in H^{1}_{0,\rho}\left(a,b;H\right)$ we can simply appeal to the fundamental theorem of calculus, cf. \cite[Cor.~6.3.7]{Cohn1980}, which is applicable, because \Cref{Sobolev} guarantees absolute continuity of $u$. We obtain:
  \begin{equation*}
    u(t) - u(a) = \int_a^t u'(s)\dx[s]
  \end{equation*}
  Since $u(a)=0$, this simply reads: $u=I_{\rho}u'$ and $u'\in L_{2,\rho}\left(a,b;H\right)$ because $u\in H^{1}_{\rho}\left(a,b;H\right)$.
\end{proof}
This result will allow us to easily transform the problem (\ref{FDE}) into a fixed point problem by simply applying the inverse time derivative to the differential equation.

\subsection{The evaluation operator}
In our examples we will encounter the evaluation operator
\begin{equation*}
  \begin{split}
    \operatorname{ev}\colon\,H^{1}_{\rho}(0,T;H)\times [0,T]&\rightarrow H\\
    (u,t)&\mapsto u(t)\,\text{.}
  \end{split}
\end{equation*}
We can observe the following immediate result:
\begin{lemma}
  \label{evHolder}
  $\operatorname{ev}$ is H\"{o}lder-continuous with exponent $\nicefrac{1}{2}$ with respect to the second argument, i.e. for given $u\in H^{1}\left(0,T;H\right)$ and $s, t \in [0,T]$ we can estimate:
  \begin{equation*}
    \norm{u(s)-u(t)}_{H} \leq \abs{s-t}^{\nicefrac{1}{2}}\norm{u}_{H^{1}(0,T;H)}
  \end{equation*}
\end{lemma}
\begin{proof}
  For $u\in H^{1}\left(0,T;H\right)$ and $0\leq s\leq t \leq T$ we can estimate:
  \begin{align*}
    \norm{u(s)-u(t)}_{H}^{2}&=\norm[\bigg]{\int_s^t u'(r) \dx[r]}_{H}^{2}\\
                            &\leq \left( \int_s^t \norm{u'(r)}_{H} \dx[r] \right)^{2}\\
                            &\leq \abs{t-s}\int_s^t \norm{u'(r)}_{H}^{2} \dx[r]\\
                            &\leq \abs{t-s}\int_0^T \norm{u'(r)}_{H}^{2} \dx[r]\\
                            &=\abs{t-s}\norm{u'}_{L_{2}(0,T;H)}^{2}\\
                            &\leq \abs{t-s}\norm{u}_{H^{1}(0,T;H)}^{2} \qedhere
  \end{align*}
\end{proof}
For $u\in H^{1}_{\rho}\left(0,T;H\right)$, due to equivalence of norms, a constant appears on the rhs. For later purposes, ideally we want Lipschitz-continuity though. To this end -- and for a subsequent projection argument -- we introduce the following subspace of $H^{1}(-h,0;H)$ for some given $\alpha>0$:
\begin{equation*}
  V_{\alpha}\coloneq \{\varphi\in H^{1}(-h,0;H)\colon\,\norm{u'}_{\infty}\leq\alpha\}\,\text{,}
\end{equation*}
where for $f\in L_{2}(a,b;H)$ we set $\norm{f}_{\infty}\coloneq \operatorname{ess}\sup_{t\in (-h,0)}\norm{f(t)}_{H}$.
\begin{remark}[$V_{\alpha}$ is a closed convex subspace of $H^{1}(-h,0;H)$]
  \label{Valphaclosed}
  Convexity is obvious. For closedness consider $\{u_{n}\}_{n\in \mathbb{N}}\subseteq V_{\alpha}$ converging to some $u\in H^{1}(-h,0;H)$. In particular $\{u'_{n}\}_{n}$ converges to $u'$ in $L_{2}(-h,0;H)$. Appealing to the Fischer-Riesz theorem, a subsequence $\{u'_{n_{k}}\}_{k}$ converges pointwise almost everywhere to the limit $u'$. Therefore $\operatorname{ess}\operatorname{sup}\abs{u'}\leq \alpha$.
\end{remark}
\begin{remark}[$\operatorname{ev}$ is Lipschitz in time on $V_{\alpha}$]
  \label{evLip}
  We can trivially estimate the evaluation mapping as a map
  \begin{align*}
    \operatorname{ev}\colon\, V_{\alpha}\times [0,T]&\rightarrow H\\
    (u,t)&\mapsto u(t)
  \end{align*}
  by computing for $s\leq t$:
  \begin{equation*}
    \norm{u(s)-u(t)}_{H}=\norm[\bigg]{\int_s^t u'(r) \dx[r]}_{H}
                            \leq \int_s^t \norm{u'(r)}_{H} \dx[r]
                            \leq \abs{t-s}\alpha\,\text{,}
  \end{equation*}
  hence verifying Lipschitz-continuity with respect to time.
\end{remark}

\section{Main results}
The main theorem of our explication is a generalized version of the Picard--Lindel\"of theorem. The conventional Picard--Lindel\"of theorem for ODEs in Hilbert spaces (cf. \cite[Thm.~4.2.3]{Waurick2022}) is not applicable in our scenario, because the evaluation mapping is not Lipschitz-continuous in general. Restricting to the subspaces $V_{\alpha}$, by \Cref{evLip} we can recover Lipschitz-continuity. We adjust our definition of Lipschitz-continuity for the rhs as follows:
\begin{definition}
  Let $U\subseteq H^{1}(-h,0;H)$ be open. We call $G\colon\,[0,T]\times U\to H$ {\em almost uniformly Lipschitz-continuous} if for all $\alpha >0$ there exists some $L_{\alpha}>0$ such that
  \begin{equation*}
    \forall t\in [0,T]\,\forall \varphi,\psi\in V_{\alpha}\colon\,\norm{G(t,\varphi)-G(t,\psi)}_{H}\leq L\norm{\varphi -\psi}_{H^{1}(-h,0;H)}
  \end{equation*}
\end{definition}
The strategy of the proof for our Picard--Lindel\"of theorem is to adapt the classical proof by Morgenstern (cf. \cite{Morgenstern1952}) utilizing exponentially weighted spaces to this generalized notion of Lipschitz-continuity. As a first step, we transform the problem (\ref{FDE}) into a fixed point problem. For that purpose we make the ansatz $u=v+\hat{\Phi}$, where
\begin{equation*}
  \hat{\Phi}\colon\,[-h,T]\to H \qquad t\mapsto
  \begin{cases}\Phi(0) &\text{for} \;\;\;\;\,0<t\leq T\\ \Phi(t) &\text{for} -h\leq t\leq 0\end{cases}
\end{equation*}
and $v\in H^{1}_{0,\rho}\left(-h,T;H\right)$. It is clear that $\hat{\Phi}\in H^{1}(-h,T;H)$, since it is absolutely continuous with weak derivatives in $L_{2}$ on both sides of $0$. We can then transform
\begin{alignat}{3}
  \begin{aligned}
  u'(t)&=G(t,u_{t})&&\text{for }\;\;\;\;\,0<t\leq T\\
  u&=\Phi&&\text{for }-h\leq t \leq 0
  \end{aligned}
  \tag{\ref{FDE}}
\end{alignat}
for $u\in H^{1}_{\rho}(-h,T;H)$ into the {\em fixed point problem (FPP)}
\begin{equation}
  \label{FPP}
  v=I_{\rho}G\bigl(\argdot,(v+\hat{\Phi})_{(\argdot)}\bigr)
\end{equation}
for $v\in H^{1}_{0,\rho}\left(-h,T;H\right)$. We prove equivalence of the problems:
\begin{lemma}
  The IVP (\ref{FDE}) and the FPP (\ref{FPP}) are equivalent.
\end{lemma}
\begin{proof}
  We quickly verify both implications:
  \begin{itemize}[leftmargin=3ex]
    \item If $u\in H^{1}_{\rho}\left(-h,T;H\right)$ is a solution of \Cref{FDE}, then we can define $v\coloneq u-\hat{\Phi}$, which is in $H^{1}_{0,\rho}\left(-h,T;H\right)$, since it is weakly differentiable and $v|_{[-h,0]}=0$. Integration of the differential equation then proves the claim.
    \item If $v\in H^{1}_{0,\rho}\left(-h,T;H\right)$ satisfies \Cref{FPP}, let
          \begin{equation*}
            u\colon\,[-h,T]\to H,\qquad t\mapsto \begin{cases}v(t)+\Phi(0)&\text{for} \;\;\;\;\,0<t\leq T\\ \Phi(t) &\text{for} -h\leq t\leq 0\end{cases}\text{.}
          \end{equation*}
          $u$ satisfies the initial condition and because $v(0)=0$, $v\in \operatorname{dom}\mathring{\partial_{t}}$ by \Cref{domainDerivative}, which after differentiation for almost all $t>0$ gives:
          \begin{equation*}
            u'(t)=(v+\Phi(0))'(t)=v'(t)=G(t,(v+\hat{\Phi})_{t})=G(t,u_{t}) \qedhere
          \end{equation*}
  \end{itemize}
\end{proof}
To make use of the subspaces $V_{\alpha}$ we introduce the metric projection on $V_{\alpha}$:
\begin{align*}
  \pi_{\alpha}\colon \, H^{1}(-h,0;H)&\to V_{\alpha}\\
  \varphi&\mapsto \begin{dcases*}\varphi & if $\varphi\in V_{\alpha}$\\
    \operatorname{argmin}\{\norm{\varphi-\psi}_{H^{1}(-h,0;H)}\colon\,\psi\in V_{\alpha}\} & if $\varphi \notin V_{\alpha}$\end{dcases*}
\end{align*}
\begin{remark}[properties of $\pi_{\alpha}$]
  We remind the reader that the best-approximation of a closed and convex subset (which $V_{\alpha}$ is by \Cref{Valphaclosed}) is unique in Hilbert spaces, cf. \cite[Thm.~V.3.2]{Werner2011}. The associated projection is in fact a Lipschitz-continuous map with Lipschitz-constant $1$.
\end{remark}
The projection allows us to pose a fixed point problem on $V_{\alpha}$. We define:
\begin{align*}
            \Gamma_{\alpha}\colon\,H^{1}_{0,\rho}\left(-h,T;H\right)&\to H^{1}_{0,\rho}\left(-h,T;H\right)\\
            v&\mapsto I_{\rho}G\left(\argdot,\pi_{\alpha}\bigl((v+\hat{\Phi})_{(\argdot)}\bigr)\right)
\end{align*}
\vspace{-3ex}
\begin{remark}[$\Gamma_{\alpha}$ is well defined]
  \label{Gammaselfmapping}
  We disperse all concerns in order:
  \begin{itemize}[leftmargin=3ex]
    \item First note that $\hat{\Phi}\in H^{1}_{\rho}\left(-h,T;H\right)$, since the derivative of $\hat{\Phi}$ exists almost everywhere because of absolute continuity and is in $L_{2,\rho}(-h,T;H)$, hence the weak derivative is as well.
    \item We know that $G$ is continuous and we can write the map $t\mapsto G\bigl(t,\pi_{\alpha}\bigl((v+\hat{\Phi})_{t}\bigr)\bigr)$ as $t\mapsto G\bigl(t,\pi_{\alpha}\bigl(\Theta_{t}(v+\hat{\Phi})\bigr)\bigr)$, which is continuous as composition of continuous maps, note \Cref{ThetaTimeCont} and \Cref{Sobolev}. Therefore the integral is well-defined.
    \item $\Gamma_{\alpha}u\in H^{1}_{0,\rho}\left(-h,T;H\right)$, since integrals over continuous functions on finite intervals are differentiable. The zero boundary term is clear by definition.
  \end{itemize}
\end{remark}
The modified FPP is then to find a unique $v\in H^{1}_{0,\rho}\left(-h,T;H\right)$ such that
\begin{equation}
  \label{FPPProj}
  v=\Gamma_{\alpha}v
\end{equation}
We can prove that it is equivalent to the modified IVP
\begin{alignat}{4}
  \label{FDEProj}
  \begin{aligned}
    u'(t)&=G\left(t,\pi_{\alpha}\left((u)_{t}\right)\right) &&\quad\text{for }&&&t>0\,\text{,}\\
    u(t)&=\Phi(t)&&\quad\text{for }&&&-h\leq t\leq 0\,\text{.}
  \end{aligned}
\end{alignat}
\begin{lemma}
  \label{EquivModified}
  Let $\alpha >0$ and $\Phi \in V_{\alpha}$. Let $U\subseteq H^{1}(-h,0;H)$ be an open neighbourhood of $\Phi$. Let $G\colon\,[0,T]\times U\to H$ be continuous. Then the following are equivalent:
  \begin{enumerate}[label=\roman*), leftmargin=4ex]
    \item \Cref{FDEProj} admits a unique solution.
    \item \Cref{FPPProj} admits a unique solution.
  \end{enumerate}
  In either case, the solutions coincide.\footnote{Using a Gr\"onwall argument, one can even show, that it is exponentially bounded.}
\end{lemma}
\begin{proof}
  \begin{itemize}[leftmargin=2ex]
    \item Let $u$ be a unique local solution of \Cref{FDEProj}. We integrate the differential \cref{FDEProj} and obtain:
          \begin{equation*}
            u(t)-u(0)= \int_0^t G\bigl(s,\pi_{\alpha}(u_{s})\bigr)\dx[s]
          \end{equation*}
          We substitute $v\coloneq u-\hat{\Phi}$ and immediately see that $v\in H^{1}_{0,\rho}\left(-h,T;H\right)$ and plugging $v$ into the previous equation yields for $t>0$:
          \begin{equation*}
            v(t)=\int_0^t G\bigl(s,\pi_{\alpha}((v+\hat{\Phi})_{s})\bigr)\dx[s]
          \end{equation*}
          Hence we obtain a fixed point of \Cref{FPPProj}. If the fixed point is not unique then because $v$ is weakly differentiable (continuously differentiable for $t>0$, weakly for $t<0$ and continuous), out of two fixed points $v$ and $w$ we obtain two functions $u_{v}\coloneq v-\hat{\Phi}$ and $u_{w}\coloneq w-\hat{\Phi}$ and consequently we are able to identify two different solutions to \Cref{FDEProj}.
    \item The converse follows by applying the same arguments in reverse direction.
  \end{itemize}
  That the solutions agree is apparent from our deduction.
\end{proof}
Utilizing the projection, we obtain at least local existence and uniqueness:
\begin{theorem}[Picard--Lindel\"of for DDEs -- local version]
  \label{PLlocal}
  Let $\alpha >0$ and $\Phi \in V_{\alpha}$. Let $U\subseteq H^{1}(-h,0;H)$ be an open neighbourhood of $\Phi$. Let $G\colon\,[0,T]\times U\to H$ be almost uniformly Lipschitz-continuous. Then the IVP (\ref{FDE}) admits a unique local solution.
\end{theorem}
\begin{proof}
  We show that the IVP (\ref{FDEProj}) admits a global solution first. Appealing to the previous lemma, for this we need to show that the FPP (\ref{FPPProj}) has a unique solution for large enough $\rho$. We do this by invoking the contraction mapping principle. We have already shown in \Cref{Gammaselfmapping} that $\Gamma_{\alpha}$ is a self-mapping, we only need to verify that it is a contraction. We estimate for $u,v \in H^{1}_{0,\rho}\left(-h,T;H\right)$:
  \begin{align*}
    &\norm[\big]{\Gamma_{\alpha}v-\Gamma_{\alpha}w}^{2}_{H^{1}_{\rho}\left(-h,T;H\right)}
     =\norm[\big]{I_{\rho} \bigl[G\bigl(.,\pi_{\alpha}\bigl((v-\hat{\Phi})_{(\argdot)}\bigr)\bigr) - G\bigl(.,\pi_{\alpha}\bigl((w-\hat{\Phi})_{(\argdot)}\bigr)\bigr) \bigr]}^{2}_{H^{1}_{\rho}}\\
    &\qquad\leq \norm{I_{\rho}}_{L_{2,\rho}\rightarrow H_{\rho}^{1}}^{2} \norm[\big]{G\bigl(.,\pi_{\alpha}\bigl((v-\hat{\Phi})_{(\argdot)}\bigr)\bigr) - G\bigl(.,\pi_{\alpha}\bigl((w-\hat{\Phi})_{(\argdot)}\bigr)\bigr)}^{2}_{L_{2,\rho}}\\
    &\qquad\leq \left(1+\tfrac{1}{\rho^{2}}\right) \int_0^T \norm[\big]{G\bigl(t,\pi_{\alpha}\bigl((v-\hat{\Phi})_{t}\bigr)\bigr) -
      G\bigl(t,\pi_{\alpha}\bigl((w-\hat{\Phi})_{t}\bigr)\bigr)}_{H}^{2}\e^{-2\rho t}\dx[t]\\
    &\qquad\leq \left(1+\tfrac{1}{\rho^{2}}\right) \int_0^T L_{\alpha}^{2}\norm[\big]{\pi_{\alpha}((v-\hat{\Phi})_{t}) -
      \pi_{\alpha}((w-\hat{\Phi})_{t})}_{H^{1}(-h,0;H)}^{2}\e^{-2\rho t}\dx[t]\\
    &\qquad\leq \left(1+\tfrac{1}{\rho^{2}}\right)L_{\alpha}^{2} \int_0^T \norm{(v-\hat{\Phi})_{t} - (w-\hat{\Phi})_{t}}_{H^{1}(-h,0;H)}^{2}\e^{-2\rho t}\dx[t]\\
    &\qquad= \left(1+\tfrac{1}{\rho^{2}}\right)L_{\alpha}^{2} \int_0^T \norm{v_{t} - w_{t}}_{H^{1}(-h,0;H)}^{2}\e^{-2\rho t}\dx[t]\\
    &\qquad= \left(1+\tfrac{1}{\rho^{2}}\right)L_{\alpha}^{2} \norm[\big]{\Theta v - \Theta w}^{2}_{L_{2,\rho}(0,T;H^{1}(-h,0;H))}\\
    &\qquad\leq \left(1+\tfrac{1}{\rho^{2}}\right)L_{\alpha}^{2} \tfrac{1}{2\rho}\norm{v-w}_{H^{1}_{\rho}\left(-h,T;H\right)}^{2}
  \end{align*}
  For $\rho$ large enough this makes $\Gamma_{\alpha}$ a contraction. Hence we obtain a fixed point of \Cref{FPPProj} and therefore a solution of \Cref{FDEProj} by \Cref{EquivModified}. Since a solution of \Cref{FDEProj} can be differentiated weakly and on $[0,T]$ that derivative is continuous by virtue of it satisfying \Cref{FDEProj}, we infer that up to some positive time $T_{0}\leq T$ the solution satisfies $\norm{u'|_{[-h,T_{0}]}}_{\infty}\leq \alpha$. Consequently the metric projection is simply the identity up to $T_{0}$. This proves local existence, uniqueness is assured by uniqueness of the fixed point.
\end{proof}
As for the case of ODEs without delay we can state a global version of the theorem:
\begin{theorem}[Picard--Lindel\"of for DDEs -- global version]
  \label{PLglobal}
  The unique local solution $v$ of \Cref{FDE} from \Cref{PLlocal} can be extended to some maximal $T_{0}\in \mathbb{R}_{+}$ such that $\underset{t\uparrow T_{0}}{\operatorname{lim}}\norm{v'|_{(-h,t)}}_{\infty}=\infty$ or already $T_{0}=T$.
\end{theorem}
\begin{proof}
  We follow the usual strategy: By fixing some initial $\alpha_{1}>\norm{\Phi'}_{\infty}$ we can appeal to \Cref{PLlocal} to obtain a solution $x_{1}$ up some time $t^{\star}_{1}\leq T$. Then we can increase to $\alpha_{2}\coloneq 2\alpha_{1}$ and obtain a solution $x_{2}$ up to some time $t^{\star}_{1}\leq t^{\star}_{2}\leq T$. Note that $x_{2}|_{[-h,t^{\star}_{1}]}=x_{1}$. Indeed, since $t^{\star}_{1}$ was such that the unique solution of \Cref{FDE} stays in $V_{\alpha_{1}}$ and since $x_{2}$ solves \Cref{FDE} locally as well, $\pi_{\alpha_{1}}\left(x_{2}|_{[-h,t^{\star}_{1}]}\right)=x_{2}|_{[-h,t^{\star}_{1}]}$ and by unique solvability $x_{2}|_{[-h,t^{\star}_{1}]}=x_{1}$. We can iterate this process with $\alpha_{n}\coloneq 2^{n-1}\alpha_{1}$ and obtain $T_{0}\coloneq \lim_{n \to \infty} t_{n}^{\star}$ as monotone limit of positive numbers and a unique solution $x$ defined as $x(t)\coloneq x_{n}(t)$ for $t\in [-h,t_{n}^{\star}]$, which is well-defined by the preceeding argument.\\
  We move on to the blow-up part. We assume that $\lim_{n \to \infty} \norm{x'|_{[-h,t_{n}^{\star}]}}_{\infty}\eqcolon M<\infty$. In this case we obtain a maximal solution $x$ as above and need to show that $x\in H^{1}_{\rho}\left(-h,T;H\right)$. We estimate:
  \begin{equation*}
    \norm{x'}_{L_{2,\rho}(-h,T;H)}^{2}= \int_{-h}^T \norm{x'(t)}^{2}_{H}\e^{-2\rho t}\dx[t]
                                    \leq M^{2}\int_{-h}^T \e^{-2\rho t}\dx[t]
                                    \leq \tfrac{M^{2}}{2\rho}\e^{2\rho h}
  \end{equation*}
  and
  \begin{align*}
    &\norm{x}_{L_{2,\rho}(-h,T;H)}^{2}= \int_{-h}^T \norm{x(t)}^{2}_{H}\e^{-2\rho t}\dx[t]\\
    &\qquad= \int_{-h}^0 \norm{x(t)}_{H}^{2}\e^{-2\rho t} \dx[t] + \int_0^T \norm{x(t)}_{H}^{2} \e^{-2\rho t} \dx[t]\\
    &\qquad= \int_{-h}^0 \norm{x(t)}_{H}^{2}\e^{-2\rho t} \dx[t] + \int_0^T \norm[\bigg]{x(0) + \int_0^t x'(s) \dx[s]}_{H}^{2} \e^{-2\rho t} \dx[t]\\
    &\qquad\leq \norm{\Phi}_{L_{2}(-h,0;H)}^{2} \e^{2\rho h} + 2 \norm{x(0)}_{H}^{2} \int_{0}^T \e^{-2\rho t} \dx[t] + 2 \int_{0}^T \norm[\bigg]{\int_{0}^t x'(s) \dx[s]}_{H}^{2}\e^{-2\rho t} \dx[t]\\
    &\qquad\leq \norm{\Phi}_{L_{2}(-h,0;H)}^{2} \e^{2\rho h} + \tfrac{\norm{x(0)}_{H}^{2}}{\rho} \e^{2\rho h} + 2 \int_{0}^T \norm{(I_{\rho}x')(t)}_{H}^{2} \e^{-2\rho t} \dx[t]\\
    &\qquad= \left(\norm{\Phi}_{2}^{2}+\tfrac{\norm{x(0)}_{H}^{2}}{\rho}\right)\e^{2\rho h} + 2 \norm{I_{\rho}x'}_{L_{2,\rho}(0,T;H)}^{2}\\
    &\qquad\leq \left(\norm{\Phi}_{2}^{2}+\tfrac{\norm{x(0)}_{H}^{2}}{\rho}\right)\e^{2\rho h} + \tfrac{2}{\rho^{2}} \norm{x'}_{L_{2,\rho}(0,T;H)}^{2}
  \end{align*}
  All terms in the last line are finite.\\
  Conversely, if $T_{0}<T$ and $x\notin H^{1}(-h,T;H)$ we can verify $\lim_{n \to \infty} \norm{x'|_{[-h,t_{n}^{\star}]}}_{\infty}=\infty$. This is but trivial, because by definition we chose $t_{n}^{\star}$ maximal such that $x'|_{[-h,t_{n}^{\star}]}\in V_{\alpha_{n}}$ and therefore if the limit does not terminate in $T$, the desired limit needs to be $+\infty$.
\end{proof}
To finish the discussion of well-posedness we elaborate that the solutions depend continuously on initial data and the right-hand side:
\begin{theorem}[continuous dependence on initial datum]
  Let $\alpha>0$. Then there exists $C>0$ such that for any given $\Phi,\Psi \in V_{\alpha}$ the corresponding solutions $u$ and $v$ of the IVP \Cref{FDE} to initial data $\Phi$ and $\Psi$ respectively exist up to some common time $T_{0}\leq T$. For any $T_{1}<T_{0}$ the restrictions satisfy $u|_{[-h,T_{1}]},v|_{[-h,T_{1}]}\in H^{1}_{\rho}(-h,T_{1};H)$ and the estimate
  \begin{equation*}
    \norm{u-v}_{H^{1}_{\rho}(-h,T_{1};H)}\leq C\norm{\Phi-\Psi}_{H^{1}(-h,0;H)}\,\text{.}
  \end{equation*}
\end{theorem}
\begin{proof}
  The first part of the statement is simply an application of \Cref{PLglobal}. To obtain the estimate we recall that being a solution to \Cref{FDE} is equivalent to $u-\hat{\Phi}$ and $v-\hat{\Psi}$ being solutions to the FPP (\ref{FPP}). We estimate thusly:
  \begin{equation*}
    \norm{u-v}_{H^{1}_{\rho}\left(-h,T_{1};H\right)}^{2}= \norm{\hat{\Phi}-\hat{\Psi}}^{2}_{H^{1}_{\rho}\left(-h,T_{1};H\right)} + \norm{(u-\hat{\Phi})-(v-\hat{\Psi})}_{H^{1}_{\rho}\left(0,T_{1};H\right)}^{2}
  \end{equation*}
  The first term can be estimated as follows:
  \begin{align*}
    \norm{\hat{\Phi}-\hat{\Psi}}^{2}_{H^{1}_{\rho}\left(-h,T_{1};H\right)} &\leq \norm{\Phi-\Psi}^{2}_{H^{1}_{\rho}(-h,0;H)} + \norm{\Phi(0)-\Psi(0)}_{L_{2,\rho}(0,T_{1};H)}^{2}\\
    &\leq \norm{\Phi-\Psi}^{2}_{H^{1}_{\rho}(-h,0;H)} + \norm{\Phi-\Psi}_{\infty}^{2} \int_0^{T_{1}} \e^{-2\rho t} \dx[t]\\
    &\leq \norm{\Phi-\Psi}^{2}_{H^{1}_{\rho}(-h,0;H)} + C \norm{\Phi-\Psi}^{2}_{H^{1}_{\rho}(-h,0;H)} \tfrac{1}{2\rho}\bigl(1 - \e^{-2\rho T_{1}}\bigr)\\
    &\leq \bigl(1+\tfrac{C}{2\rho}\bigr)\norm{\Phi-\Psi}^{2}_{H^{1}_{\rho}(-h,0;H)}
  \end{align*}
  For the second term we estimate
  \begin{align*}
    &\norm{(u-\hat{\Phi})-(v-\hat{\Psi})}_{H^{1}_{\rho}\left(0,T_{1};H\right)}^{2}
    = \norm[\big]{I_{\rho}\bigl[G\bigl(\argdot,\Theta_{(\argdot)}u\bigr) - G\bigl(\argdot,\Theta_{(\argdot)}v\bigr)\bigr]}_{H^{1}_{\rho}\left(0,T_{1};H\right)}^{2}\\
    &\qquad\qquad\leq \norm{I_{\rho}}_{L_{2,\rho}\to H^{1}_{\rho}}^{2} \norm[\big]{G\bigl(\argdot,\Theta_{(\argdot)}u\bigr) - G\bigl(\argdot,\Theta_{(\argdot)}v\bigr)}_{L_{2,\rho}\left(0,T_{1};H\right)}^{2}\\
    &\qquad\qquad\leq \bigl(1+\tfrac{1}{\rho^{2}}\bigr) \norm[\big]{L_{\alpha}\norm{u_{(\argdot)}- v_{(\argdot)}}_{H^{1}\left(-h,0;H\right)}}_{L_{2,\rho}\left(0,T_{1};H\right)}^{2}\\
    &\qquad\qquad= \bigl(1+\tfrac{1}{\rho^{2}}\bigr) L_{\alpha}^{2} \norm{\Theta u - \Theta v}_{L_{2,\rho}\left(0,T_{1};H^{1}_{\rho}\left(-h,0;H\right)\right)}^{2}\\
    &\qquad\qquad\leq \tfrac{L_{\alpha}^{2} \left(1+\tfrac{1}{\rho^{2}}\right)}{2\rho}\norm{u-v}_{H^{1}_{\rho}\left(-h,T_{1};H\right)}^{2}
  \end{align*}
  Note that we can make the initial calculations since by truncating the functions to $T_{1}$ they remain in $V_{\alpha}$ for some $\alpha>0$. Putting these two estimates together, demanding that $\rho$ be large enough and rearranging provides us with the inequality
  \begin{equation*}
    \norm{u-v}_{H^{1}_{\rho}\left(-h,T_{1};H\right)}\leq \tilde{C}\norm{\Phi-\Psi}_{H^{1}_{\rho}\left(-h,0;H\right)}\,\text{.}
  \end{equation*}
  Finally, appealing to the equivalence of the weighted and unweighted norms produces the desired result for all $\rho$.
\end{proof}
The following theorem on dependence on the rhs is in the spirit of \cite[Thm.~3.6 and Rem.~3.7]{Kalauch2014}.
\begin{theorem}[continuous dependence on the rhs]
  Let $\alpha>0$. Then there exists $C>0$ such that for any $\Phi\in H^{1}(-h,0;H)$, any open neighbourhood $U$ of $\Phi$ and almost uniformly Lipschitz-continuous functions $F,G\colon [0,T]\times U\to H$ there exist solutions $u$ and $v$ to the IVP (\ref{FDE}) to rhs $F$ and $G$ respectively up to some common time $T_{0}\leq T$. For any $T_{1}<T_{0}$ the restrictions satisfy $u|_{[-h,T_{1}]},v|_{[-h,T_{1}]} \in H^{1}_{\rho}\left(-h,T;H\right)$ and the estimate
  \begin{equation*}
    \norm{u-v}_{H^{1}_{\rho}(-h,T_{1};H)}\leq C\sup_{t\in [0,T_{1}]}\norm{F(t,u_{t})-G(t,u_{t})}_{H}\,\text{.}
  \end{equation*}
\end{theorem}
\begin{proof}
  The first part of the statement is clear: We simply apply \Cref{PLglobal} to show that the restrictions are in $H^{1}_{\rho}\left(-h,T_{1};H\right)$. To obtain the estimate we recall that being a solution to \Cref{FDE} is equivalent to $u-\hat{\Phi}$ and $v-\hat{\Psi}$ being solutions of the fixed point \cref{FPP}. We can therefore estimate:
  \begin{align*}
    &\norm{u-v}_{H^{1}_{\rho}\left(-h,T_{1};H\right)}^{2}= \norm[\big]{I_{\rho}\bigl[F\bigl(\argdot,\pi_{\alpha}(u_{(\argdot)})\bigr)
                                                      - G\bigl(\argdot,\pi_{\alpha}(v_{(\argdot)})\bigr)\bigr]}_{H^{1}_{\rho}\left(0,T_{1};H\right)}^{2} \\
    &\qquad\leq \bigl(1+\tfrac{1}{\rho^{2}}\bigr) \norm[\big]{F\bigl(\argdot,\pi_{\alpha}(u_{(\argdot)})\bigr)
      - G\bigl(\argdot,\pi_{\alpha}(v_{(\argdot)})\bigr)}_{L_{2,\rho}\left(0,T_{1};H\right)}^{2} \\
                                                    &\qquad= \bigl(1+\tfrac{1}{\rho^{2}}\bigr) \int_0^{T_{1}} \norm[\big]{F\bigl(t,\pi_{\alpha}(u_{t})\bigr)
                                                      - G\bigl(t,\pi_{\alpha}(v_{t})\bigr)}_{H}\e^{-2\rho t}\dx[t] \\
    &\qquad\leq \bigl(1+\tfrac{1}{\rho^{2}}\bigr) \Bigl[ \int_0^{T_{1}} \norm[\big]{F\bigl(t,\pi_{\alpha}(u_{t})\bigr)
      - G\bigl(t,\pi_{\alpha}(u_{t})\bigr)}_{H}^{2}\e^{-2\rho t}\dx[t] \Bigr.\\
    &\qquad\qquad + \Bigl. L_{G}^{2}\int_0^{T_{1}} \norm{\pi_{\alpha}(u_{t})
      - \pi_{\alpha}(v_{t})}_{H^{1}(-h,0;H)}^{2}\e^{-2\rho t}\dx[t] \Bigr]
  \end{align*}
  We can estimate the second term in the usual fashion:
  \begin{align*}
    \int_0^{T_{1}} \norm{\pi_{\alpha}(u_{t})
    - \pi_{\alpha}(v_{t})}_{H^{1}}^{2}\e^{-2\rho t}\dx[t]
    &\leq \int_0^{T_{1}} \norm{u_{t} - v_{t}}_{H^{1}(-h,0;H)}^{2}\e^{-2\rho t}\dx[t]\\
    &= \norm{\Theta u - \Theta v}_{L_{2,\rho}(-h,T_{1};H^{1}(-h,0;H))}^{2} \\
    &\leq \tfrac{1}{2\rho} \norm{u-v}_{H^{1}_{\rho}(-h,T;H)}^{2}
  \end{align*}
  The first term we can estimate using the fact that $t\mapsto u_{t}$ is a continuous function by \Cref{ThetaTimeCont}:
  \begin{align*}
    \int_0^{T_{1}} \norm[\big]{F\bigl(t,\pi_{\alpha}(u_{t})\bigr)
      - G\bigl(t,\pi_{\alpha}(u_{t})\bigr)}_{H}^{2}\e^{-2\rho t}\dx[t] \leq \tfrac{1}{2\rho} \sup_{t\in [0,T_{1}]}\norm[\big]{F(t,u_{t})-G(t,u_{t})}_{H}^{2}
  \end{align*}
  Hence:
  \begin{align*}
    \norm{u-v}_{H^{1}_{\rho}\left(-h,T_{1};H\right)}^{2}&\leq C_{1}\sup_{t\in [0,T_{1}]}\norm{F(t,u_{t})-G(t,u_{t})}_{H}^{2} + C_{2}L_{G}^{2}\norm{u-v}_{H^{1}_{\rho}(-h,T;H)}^{2}
  \end{align*}
  We now can make $\rho$ large enough to force $C_{2}<1$. Rearranging the inequality yields:
  \begin{equation*}
    \norm{u-v}_{H^{1}_{\rho}\left(-h,T_{1};H\right)}\leq C \sup_{t\in [0,T_{1}]}\norm{F(t,u_{t})-G(t,u_{t})}_{H} \qedhere
  \end{equation*}
\end{proof}

\section{Applications}
We outline three applications of our selected results. An additional example warranting its own article is dicussed in \cite{Aigner2024}.

\subsection{Constant delay}
A trivial case, that is nevertheless worth mentioning is the case of constant delay. Problems of the form
\begin{alignat*}{4}
    x'(t)&=G(t,x(t-t_{1}),\dots,x(t-t_{n})) &&\quad\text{for } &&& 0<t\leq T\,\text{,}\\
    x(t)&=\Phi &&\quad\text{for } &&& -h\leq t\leq 0\>\,\text{,}
\end{alignat*}
where $0\leq t_{1}\leq \dots \leq t_{n}\leq h$ have a unique solution provided that $G\colon [0,T]\times U^{n}\to H$ is locally Lipschitz-continuous with respect to the $H^{n}$-variable and that the prehistory satisfies $\norm{\Phi'}_{\infty}<\infty$ according to \Cref{PLglobal}. We simply observe that the rhs can be written as $G(t,r(u_{t}))$, where
\begin{equation*}
  r\colon\,H^{1}(-h,0;H)\to H^{n}\qquad \varphi\mapsto \left(\varphi(-t_{1}),\dots,\varphi(-t_{n})\right)^{T}\text{.}
\end{equation*}
Since the evaluation functional is Lipschitz-continuous on any $V_{\alpha}$ according to \Cref{evLip} the composition $G(\argdot,r(u_{(\argdot)}))$ is almost uniformly Lipschitz-continuous.

\subsection{An academic example}
\label{academicEx}
We consider a delay of the form $\abs{x(t)}$ in the simplest possible IVP:
\begin{align*}
  x'(t)&= x(t-\abs{x(t)}) &&\text{for}\qquad \quad\, t>0\\
  x(t)&= \Phi(t) &&\text{for}-h\leq t \leq 0
\end{align*}
The importance of this example lies in the fact, that it admits non-unique solutions for continuous initial prehistories, cf. \cite[Ex.~4.3]{Waurick2023}. This fact proves to be problematic in solution theory of DDEs, for instance in the semigroup approach to DDEs. It highlights, that the problem lies with the initial prehistories, not the equation itself. We verify, that the system satisfies the requirements of our well-posedness result \Cref{PLlocal}. If we wish, we can write the equation using the delay functional $\tau\colon\,\varphi\mapsto \abs{\varphi(0)}$. $\tau$ is a Lipschitz-continuous functional, since for $\varphi,\psi \in H^{1}(-h,0;H)$ we can estimate:
\begin{alignat*}{3}
  \abs{\tau(\varphi)-\tau(\psi)}&= \abs{\abs{\varphi(0)}-\abs{\psi(0)}} &&\leq \abs{\varphi(0)-\psi(0)}\\
                                 &\leq \norm{\varphi-\psi}_{\infty} &&\leq C \norm{\varphi-\psi}_{H^{1}(-h,0;H)}
\end{alignat*}
where the second inequality is simply the second triangle inequality. Writing the rhs of the differential equation as a function of histories, we obtain $F(\varphi)\coloneq \varphi(-\tau(\varphi))$. We can now verify that this function is almost uniformly Lipschitz-continuous:\\
For $\varphi,\psi \in V_{\alpha}$ for some $\alpha > 0$ without loss of generality $\tau(\varphi)\geq  \tau(\psi)$ and for $t>0$ we can estimate
\begin{align*}
  &\norm{\varphi(-\tau(\varphi)) - \psi(-\tau(\psi))}_{H}\\
  &\qquad= \norm[\bigg]{\varphi(-h) + \int_{-h}^{-\tau(\varphi)} \varphi'(s)\dx[s]
    - \psi(-h) - \int_{-h}^{-\tau(\psi)} \psi'(s)\dx[s]}_{H}\\
  &\qquad\leq \norm{\varphi(-h)-\psi(-h)} + \int_{-h}^{-\tau(\psi)} \norm{\varphi'(s)-\psi'(s)} \dx[s] + \int_{-\tau(\psi)}^{-\tau(\varphi)} \norm{\varphi'(s)} \dx[s]\\
  &\qquad\leq \norm{\varphi-\psi}_{\infty} + \sqrt{h}\left(\int_{-h}^0 \norm{\varphi'(s)-\psi'(s)}^{2} \dx[s]\right)^{-\nicefrac{1}{2}} + \alpha \abs{\tau(\varphi) - \tau(\psi)}\\
  &\qquad\leq \norm{\varphi - \psi}_{H^{1}(-h,0;H)}\bigl(C + \sqrt{h} + \alpha C \bigr)
\end{align*}
Appealing to \Cref{PLlocal} we obtain at least a unique local solution for given initial prehistories $\Phi$ with bounded derivative.\\
Our approach is in fact applicable for much more general cases of nonconstant delay of the form
\begin{equation*}
  x'(t) = g\bigl(t,x(t+r_{1}(x_{t})),\dots,x(t+r_{m}(x_{t}))\bigr)\,\text{,}
\end{equation*}
where the $r_{i}\colon\,H^{1}(-h,0;H)\to [-h,0]$ are assumed to be Lipschitz-continuous and $g\colon\,[0,T]\times H^{n}\to H$ is supposed to satisfy the following notion of Lipschitz-continuity: Suppose there exists $L\geq 0$ such that for all $t\in [0,T]$ and $u_{1},\dots,u_{n}$, $v_{1},\dots,v_{n}\in H$ we have
\begin{equation*}
  \norm{g(t,u_{1},\dots,u_{n})-g(t,v_{1},\dots,v_{n})}_{H}\leq L \sum_{j=1}^{n}\norm{u_{j}-v_{j}}_{H}\text{.}
\end{equation*}
Then the associated IVP for an initial prehistory $\Phi\in H^{1}(-h,0;H)$ with bounded derivative is well-posed. A proof of this statement can be found in \cite[Thm.~5.2]{Waurick2023}, where it is shown for $H= \mathbb{R}^{m}$, but the proof can be generalized effortlessly.

\subsection{Integro-differential equations}
Integro-differential equations with state- dependent delay are an important subclass and appear in many applications; for instance in models of wheel shimmy (cf. \cite{Takacs2008,Takacs2009}) and car following models (cf. \cite{Orosz2006}). They can appear in many forms, however in the spirit of the articles \cite{Khasawneh2011} and \cite{Hernandez2022}, the type of equation we are interested here is of the form:
\begin{alignat}{4}
  \label{IVPIntegroDiff}
  \begin{aligned}
    x'(t)&= F(t,x_{t})+G\Bigl(t,x_{t}, \int_0^h g(t,x_{t}(-s))k(t-s)\dx[s]\Bigr) &&\text{for}\;\;\;\;\, 0<t\leq T\\
    x(t)&= \Phi(t) &&\text{for} -h\leq t\leq 0\,\text{,}
  \end{aligned}
\end{alignat}
where
\begin{align*}
  F&\colon\,[0,T]\times H^{1}(-h,0;H) \to H\,\text{,}\\
  G&\colon\,[0,T]\times H^{1}(-h,0;H)\times H \to H\,\text{,}\\
  g&\colon\,[0,T]\times H \to H\,\text{,}\\
  k&\colon\,[-h,T]\to H\,\text{.}
\end{align*}
For the function $G$ we will assume the following Lipschitz-continuity property (G): There exists $L\geq 0$ such that for all $t\in [0,T]$ and for all $\varphi,\psi \in H^{1}(-h,0;H)$ and all $\alpha,\beta \in H$ we can estimate:
\begin{equation*}
  \norm{G(t,\varphi,\alpha)-G(t,\psi,\beta)}_{H}\leq L\left(\norm{\varphi-\psi}_{H^{1}(-h,0;H)} + \norm{\alpha - \beta}_{H}\right)
\end{equation*}
We can show:
\begin{theorem}
  For $\Phi\in H^{1}(-h,0;H)$ with bounded derivative, $F$ almost uniformly Lipschitz-continuous, $G$ satsifying property (G), $g$ Lipschitz-continuous uniformly with respect to the first variable and a kernel $k \in L^{1}(-h,T;H)$, we have a unique local solution $u \in H^{1}$ of \Cref{IVPIntegroDiff} that can be extended to some maximal $T_{0}\leq T$ such that $T_{0}=T$ or $\lim_{t\uparrow T_{0}}\norm{u'|_{(-h,t)}}_{\infty}=+\infty$.
\end{theorem}
\begin{proof}
  Appealing to \Cref{PLglobal}, we only need to verify the almost uniform Lipschitz-continuity of the rhs written as a function of $t$ and $x_{t}$. It suffices to consider the function
  \begin{equation*}
    \tilde{G}(t,\varphi)\coloneq G\Bigl(t,\varphi, \int_0^h g(t,\varphi(-s))k(t-s)\dx[s]\Bigr)\,\text{.}
  \end{equation*}
  For $\varphi,\psi\in V_{\alpha}$ for some $\alpha > 0$ we can estimate:
  \begin{align*}
    &\norm{\tilde{G}(t,\varphi)-\tilde{G}(t,\psi)}_{H}\\
    &\quad= \norm[\bigg]{G\Bigl(t,\varphi, \int_0^h g(t,\varphi(-s))k(t-s)\dx[s]\Bigr) - G\Bigl(t,\psi, \int_0^h g(t,\psi(-s))k(t-s)\dx[s]\Bigr)} \\
    &\quad\leq L_{G}\norm{\varphi-\psi}_{H^{1}(-h,0;H)} + L_{G}\norm[\bigg]{\int_0^h \bigl(g(t,\varphi(-s)) - g(t,\psi(-s))\bigr)k(t-s) \dx[s]}_{H}\\
    &\quad\leq L_{G}\norm{\varphi-\psi}_{H^{1}(-h,0;H)} + L_{G}\int_0^h \norm[\big]{g(t,\varphi(-s)) - g(t,\psi(-s))}_{H}\norm[\big]{k(t-s)}_{H} \dx[s]\\
    &\quad\leq L_{G}\norm{\varphi-\psi}_{H^{1}(-h,0;H)} + L_{G} L_{g}\int_0^h \norm{\varphi(-s) - \psi(-s)}_{H}\norm{k(t-s)}_{H} \dx[s]\\
    &\quad\leq L_{G}\norm{\varphi-\psi}_{H^{1}(-h,0;H)} + L_{G} L_{g}\norm{\varphi-\psi}_{\infty} \int_0^h \norm{k(t-s)}_{H} \dx[s]\\
    &\quad\leq L_{G}\norm{\varphi-\psi}_{H^{1}(-h,0;H)} + L_{G} L_{g} C \norm{\varphi-\psi}_{H^{1}(-h,0;H)}\norm{k}_{L^{1}(-h,T;H)}
  \end{align*}
  The last line follows from the Sobolev embedding \cref{Sobolev}.
\end{proof}

\section{Conclusion}
The approach to state-dependent DDEs presented in this article provides an alternative to the existing classical theory for delay differential equations. We used standard tools from the repertoire of the conventional weak approach to differential equations, all our theorems appear very traditional in style.\\
Established classical well-posedness results appeal to $\mathcal{C}$- and $\mathcal{C}^{1}$-functions, compactness of orbits, and the like. For instance Hale (cf. \cite{Hale1993}) shows that the IVP
\begin{alignat*}{4}
  \begin{aligned}
    x'(t)&=f(t,x_{t}) &\qquad\text{for}&&t>0\\
    x(t)&=\Phi(t)&\qquad\text{for}&&-h\leq t\leq 0
  \end{aligned}
\end{alignat*}
where $f\colon \mathbb{R}\times \mathcal{C}([-h,0];\mathbb{R}^{n})\supseteq D\to \mathbb{R}^{n}$ is a given continuous function and $\Phi\in  \mathcal{C}([-h,0];\mathbb{R}^{n})$ is a given initial prehistory, has a solution if $D$ is open and $(-h,\Phi)\in D$. Uniqueness is shown if in addition $f$ is Lipschitz-continuous with respect to the second argument on every compact subset of $D$ and the solution depends continuously on initial datum.\\
Nominally, Hale's classical results appear stronger, since the requirement for existence of solutions is mere continuity. We point out the important difference, that Hale's result is based upon Schauder's fixed point theorem, which does not guarantee uniqueness, contrary to the contraction mapping principle utilized in this article. The contraction mapping principle is also superior in the sense that it naturally provides continuous dependence on data and allows for easy numerics. Hale's result additionally comes with the severe restriction of being set in $\mathbb{R}^{n}$. The uniqueness assumption is fairly close to the one stated here, but the setting in our case is more approachable, as we do not require statements about compact subsets. Furthermore, on the space $V_{\alpha}$ we investigate, Hale's Lipschitz-condition implies ours.\\
Another classical well-posedness approach utilizes the concept of the solution manifold (cf. \cite{Walther2003}), which is a concept that was originally designed to study linearized stability of delay differential equations, but that can also be used to prove existence and uniqueness of solutions. In that context systems of the form
\begin{alignat*}{4}
  \begin{aligned}
    x'(t)&=f(x_{t}) &\qquad\text{for}&&t>0\\
    x(t)&=\Phi(t)&\qquad\text{for}&&-h\leq t\leq 0
  \end{aligned}
\end{alignat*}
are studied, where $f\colon\,\mathcal{C}^{1}([-h,0])\supseteq U\to \mathbb{R}^{n}$ for some open $U$ is supposed to satisfy the additional assumptions
\begin{itemize}[leftmargin=3ex]
  \item Each $Df(\phi)$ for $\phi \in U$ extends to a linear map $D_{e}f(\phi)$ on $\mathcal{C}([-h,0])$ which is continuous.
  \item For every $\phi \in U$ there exist an open neighbourhood $V\subseteq U$ and $L\geq 0$ such that for all $\psi,\chi \in V$ the following estimate holds:
        \begin{equation*}
          \abs{f(\psi)-f(\chi)}\leq L\norm{\psi-\chi}_{\mathcal{C}^{1}([-h,0])}
        \end{equation*}
\end{itemize}
In this setting, the first condition implies that the set
\begin{equation*}
  \mathcal{X}\coloneq \{\phi \in U\colon\,\phi'(0)=f(\phi)\}
\end{equation*}
is an $n$-dimensional $\mathcal{C}^{1}$-submanifold -- provided that it is nonempty in the first place. The second condition implies that the tangent space of $\mathcal{X}$ is nice and introduces a continuous semiflow for every prehistory $\phi\in\mathcal{X}$, i.e. the solutions of the problem.\\
Comparatively in this scenario our findings are a strict improvement in the context of solution theory. Since the assumptions in the solution manifold approach are strictly stronger though, a benefit of this article is that we prove that the solution manifold is non-empty, giving us access to this tool to study stability if we wish to do so. A more in-depth comparison of these two approaches in the case of a DDE from cell biology is contained in \cite{Aigner2024}.\\
A different approach to state-dependent DDEs is the use of semigroups (cf. \cite{Batkai2005} and the references therein). The semigroup approach comes with the disadvantage of having to introduce an extended state space, at least if the underlying Banach space is the space of continuous functions. Our approach makes due without. We point out again the importance of the regularity of the prehistories as discussed in the example in \Cref{academicEx} -- mere continuity does not suffice. The presented setting of $H^{1}$-prehistories leads to a simple well-posedness theory though. For comparisons to other approaches we also refer to the concluding discussion in \cite{Waurick2023}.\\
We would like to close by stressing that the key idea behind the presented approach is the use of exponentially weighted Sobolev spaces that grants access to a (boundedly) invertible time derivative. This idea is a foundation of the so-called theory of evolutionary equations that allows the inversion of operators of the form $\partial_{t}M(\partial_{t})+A$, where $M$ is a so-called material law and $A$ is a (time-independent) skew-adjoint operator, on a suitable $L_{2}$ space (Picard's theorem, cf. \cite[Thm.~6.2.1]{Waurick2022}). Utilizing this theory, it is possible to extend the ideas presented here to the theory to PDEs with state-dependent delay. This will be the topic of future research.

\bibliographystyle{abbrvurl}
\bibliography{references}

\begin{thebibliography}{10}

\bibitem{Adams2003}
R.~A. Adams and J.~J.~F. Fournier.
\newblock {\em Sobolev spaces}, volume 140 of {\em Pure Appl. Math., Academic
  Press}.
\newblock New York, NY: Academic Press, 2nd ed. edition, 2003.

\bibitem{Aigner2024}
B.~Aigner and M.~Waurick.
\newblock A simple way to well-posedness in $h^{1}$ of a delay differential
  equation from cell biology, 2024.
\newblock URL: \url{https://arxiv.org/abs/2406.06630}, \href
  {https://arxiv.org/abs/2406.06630} {\path{arXiv:2406.06630}}.

\bibitem{Arendt2015}
W.~Arendt, R.~Chill, C.~Seifert, H.~Vogt, and J.~Voigt.
\newblock Form methods for evolution equations, and applications.
\newblock 18th Internet Seminar on Evolution Equations, 2015.
\newblock
  \url{https://www.mat.tuhh.de/veranstaltungen/isem18/pdf/LectureNotes.pdf}.

\bibitem{Batkai2005}
A.~B{\'a}tkai and S.~Piazzera.
\newblock {\em Semigroups for delay equations}, volume~10 of {\em Res. Notes
  Math.}
\newblock Wellesley, MA: A K Peters, 2005.

\bibitem{Cohn1980}
D.~L. Cohn.
\newblock {\em Measure theory.}
\newblock Boston, MA: Birkh{\"a}user, reprinted of the orig. 1980 edition,
  1997.

\bibitem{Evans2010}
L.~C. Evans.
\newblock {\em Partial differential equations}, volume~19 of {\em Grad. Stud.
  Math.}
\newblock Providence, RI: American Mathematical Society (AMS), 2nd ed. edition,
  2010.

\bibitem{Waurick2023}
J.~Frohberg and M.~Waurick.
\newblock State-dependent delay differential equations on $h^1$, 2023.
\newblock \href {https://arxiv.org/abs/2308.04730} {\path{arXiv:2308.04730}}.

\bibitem{Hale1993}
J.~K. Hale and S.~M. Verduyn~Lunel.
\newblock {\em Introduction to functional-differential equations}, volume~99 of
  {\em Applied Mathematical Sciences}.
\newblock Springer-Verlag, New York, 1993.
\newblock \href {https://doi.org/10.1007/978-1-4612-4342-7}
  {\path{doi:10.1007/978-1-4612-4342-7}}.

\bibitem{Hernandez2022}
E.~Hernandez, V.~Rolnik, and T.~M. Ferrari.
\newblock Existence and uniqueness of solutions for abstract
  integro-differential equations with state-dependent delay and applications.
\newblock {\em Mediterr. J. Math.}, 19(3):Paper No. 101, 13, 2022.
\newblock \href {https://doi.org/10.1007/s00009-022-02009-2}
  {\path{doi:10.1007/s00009-022-02009-2}}.

\bibitem{Kalauch2014}
A.~Kalauch, R.~Picard, S.~Siegmund, S.~Trostorff, and M.~Waurick.
\newblock A {Hilbert} space perspective on ordinary differential equations with
  memory term.
\newblock {\em J. Dyn. Differ. Equations}, 26(2):369--399, 2014.
\newblock \href {https://doi.org/10.1007/s10884-014-9353-6}
  {\path{doi:10.1007/s10884-014-9353-6}}.

\bibitem{Khasawneh2011}
F.~A. Khasawneh and B.~P. Mann.
\newblock Stability of delay integro-differential equations using a spectral
  element method.
\newblock {\em Math. Comput. Modelling}, 54(9-10):2493--2503, 2011.
\newblock \href {https://doi.org/10.1016/j.mcm.2011.06.009}
  {\path{doi:10.1016/j.mcm.2011.06.009}}.

\bibitem{Morgenstern1952}
D.~Morgenstern.
\newblock {\em Beitr\"age zur nichtlinearen Funktionalanalysis}.
\newblock PhD thesis, TU Berlin, Berlin, Germany, 1952.

\bibitem{Orosz2006}
G.~Orosz and G.~St{\'e}p{\'a}n.
\newblock Subcritical hopf bifurcations in a car-following model with
  reaction-time delay.
\newblock {\em Proceedings of the Royal Society A: Mathematical, Physical and
  Engineering Sciences}, 462(2073):2643--2670, 2006.

\bibitem{Waurick2022}
C.~Seifert, S.~Trostorff, and M.~Waurick.
\newblock {\em Evolutionary equations---{P}icard's theorem for partial
  differential equations, and applications}, volume 287 of {\em Operator
  Theory: Advances and Applications}.
\newblock Birkh\"{a}user/Springer, Cham, [2022] \copyright 2022.
\newblock \href {https://doi.org/10.1007/978-3-030-89397-2}
  {\path{doi:10.1007/978-3-030-89397-2}}.

\bibitem{Takacs2009}
D.~Tak{\'a}cs, G.~Orosz, and G.~St{\'e}p{\'a}n.
\newblock Delay effects in shimmy dynamics of wheels with stretched string-like
  tyres.
\newblock {\em European Journal of Mechanics-A/Solids}, 28(3):516--525, 2009.

\bibitem{Takacs2008}
D.~Tak{\'a}cs, G.~St{\'e}p{\'a}n, and S.~J. Hogan.
\newblock Isolated large amplitude periodic motions of towed rigid wheels.
\newblock {\em Nonlinear Dynamics}, 52:27--34, 2008.

\bibitem{Walther2003}
H.-O. Walther.
\newblock The solution manifold and {{\(C^{1}\)}}-smoothness for differential
  equations with state-dependent delay.
\newblock {\em J. Differ. Equations}, 195(1):46--65, 2003.
\newblock \href {https://doi.org/10.1016/j.jde.2003.07.001}
  {\path{doi:10.1016/j.jde.2003.07.001}}.

\bibitem{Werner2011}
D.~Werner.
\newblock {\em Funktionalanalysis}.
\newblock Springer-Lehrb. Berlin: Springer, 7th revised and expanded ed.
  edition, 2011.
\newblock \href {https://doi.org/10.1007/978-3-642-21017-4}
  {\path{doi:10.1007/978-3-642-21017-4}}.

\end{thebibliography}

\end{document}